\documentclass[11pt]{article}
\usepackage{amssymb, amsmath}
\usepackage{amsthm}
\usepackage[margin=1in]{geometry}
\usepackage{enumerate}
\allowdisplaybreaks[1]
\numberwithin{equation}{section}

\newtheorem{thm}{Theorem}[section]
\newtheorem{cor}{Corollary}[section]

\newtheorem{problem}{Problem}[section]
\newtheorem{rem}{Remark}[section]

\usepackage{graphicx}
\usepackage{mathrsfs}
\begin{document}
\markboth{R. Rajkumar and P. Devi}{}
\title{Intersection graph of cyclic subgroups of groups}

\author{R. Rajkumar\footnote{e-mail: {\tt rrajmaths@yahoo.co.in}},\ \ \
P. Devi\footnote{e-mail: {\tt pdevigri@gmail.com}}\\
{\footnotesize Department of Mathematics, The Gandhigram Rural Institute -- Deemed University,}\\ \footnotesize{Gandhigram -- 624 302, Tamil Nadu, India.}\\[3mm]
}
\date{}
\maketitle

\begin{abstract}
Let $G$ be a group.
\textit{The intersection graph of cyclic subgroups of $G$}, denoted by $\mathscr I_c(G)$, is a graph having all the proper cyclic subgroups of
$G$ as its vertices and two distinct vertices  in $\mathscr I_c(G)$ are adjacent if and only if
their intersection is non-trivial. In this paper, we classify the finite groups whose intersection graph of cyclic subgroups is one of totally disconnected, complete, star, path, cycle. We show that for a given finite group $G$, $girth(\mathscr I_c (G)) \in \{3, \infty\}$. Moreover, we classify all finite non-cyclic abelian groups whose intersection graph of cyclic subgroups is planar.   Also for any group $G$, we determine the independence number, clique cover number of $\mathscr I_c (G)$ and show that $\mathscr I_c (G)$ is weakly $\alpha$-perfect. Among the other results, we determine the values of $n$ for which $\mathscr I_c (\mathbb{Z}_n)$ is regular and estimate its domination number.
\paragraph{Keywords:}Intersection graph, cyclic subgroups, girth, weakly $\alpha$-perfect, planar.
\paragraph{2010 Mathematics Subject Classification:} 05C25,  05C10,  05C17.

\end{abstract}

\section{Introduction} \label{sec:int} Associating graphs to algebraic structures and studying their properties using the methods of graph theory has been an interesting topic for mathematicians in the last decade. For a given family $\mathcal{F}=\{S_i~|~i\in I\}$ of sets, the intersection graph of $\mathcal{F}$ is a graph with the members of $\mathcal{F}$ as its vertices and two vertices $S_i$ and $S_j$ are adjacent if and only if $i\neq j$ and $S_i\cap S_j\neq \{\emptyset\}$. For the properties of these graphs and several special class of intersection graphs, we refer the reader to \cite{mckee}. In the past fifty years,
it has been a growing interest among mathematicians, when the members of $\mathcal{F}$ have some specific algebraic
structures.
In 1964 Bosak~\cite{bosak} defined the intersection graphs of semigroups. Motivated by this,
Cs$\acute{a}$k$\acute{a}$ny and Poll$\acute{a}$k ~\cite{csak} defined the intersection graph of subgroups of a finite group.

Let $G$ be a group.  \textit{The intersection graph of subgroups of $G$}, denoted by $\mathscr{I}(G)$, is a graph having all the
proper subgroups of $G$ as its vertices and two distinct vertices in $\mathscr{I}(G)$ are adjacent if and only if
the intersection of the corresponding subgroups is non-trivial.

Some properties of the intersection graphs of subgroups of finite abelian groups were studied by Zelinka in~\cite{zelinka}. Inspired by these, there are several papers appeared in the literature which have studied the intersecting graphs on algebraic structures, viz., rings and modules. See, for instance \cite{ akbari_2,  Chakara,  raj_5, Shen, yaraneri} and the references therein.

In this paper, for a given group $G$, we define
\textit{the intersection graph of cyclic subgroups of $G$}, denoted by $\mathscr I_c(G)$, is a graph having all the proper cyclic subgroups of
$G$ as its vertices and two distinct vertices  in $\mathscr I_c(G)$ are adjacent if and only if
their intersection is non-trivial. Clearly $\mathscr I_c(G)$ is a subgraph of $\mathscr I (G)$ induced by all the proper cyclic subgroups of $G$.
When $G$ is cyclic, then $\mathscr I_c(G)$ and $\mathscr I (G)$ are the same. In \cite{Chakara}, Chakrabarty \emph{et al} have studied several properties of intersection graphs of subgroups of  cyclic groups.

Now we recall some basic definitions and notations of graph theory. We use the standard terminology of graphs (e.g., see \cite{harary}).
 Let $G$ be a graph. We denote  the degree of
a vertex $v$ in $G$ by $deg(v)$.  A graph whose edge set is empty is called a \emph{null} graph or \emph{totally disconnected} graph.  $K_n$ denotes the complete graph on $n$ vertices. $K_{m,n}$ denotes the complete bipartite graph with one partition consists of $m$ vertices and other partition consists of $n$ vertices. In particular $K_{1,n}$ is called a \emph{star}. $P_n$ and $C_n$ respectively denotes the path and cycle with $n$ edges. A graph is \emph{regular} if all the vertices have the same degree. A graph is \emph{planar} if it can be drawn in a plane such that no two edges intersect except (possibly) at their end vertices. The girth of $G$, denoted by $girth(G)$, is the length of its shortest cycle, if it exist; other wise $girth(G) = \infty$. An \emph{independent set} of $G$ is a subset of $V(G)$ having no
two vertices are adjacent. The \emph{independence number} of $G$, denoted by $\alpha(G)$, is the
cardinality of the largest independent set. A \emph{clique} of  $G$ is a complete subgraph of $G$. The \emph{clique cover number} of $G$, denoted by $\theta(G)$, is the minimum number of cliques in $G$ which cover all the vertices of $G$. $G$ said to be \emph{weakly $\alpha$-perfect} if $\alpha(G)=\theta(G)$. A \emph{dominating set} of $G$  is a subset $D$ of $V(G)$ such that every vertex not in $D$ is adjacent to at least one member of $D$. The \emph{domination number} $\gamma(G)$ is the number of vertices in a smallest dominating set for $G$.

In this paper, we classify the finite groups whose intersection graph of cyclic subgroups is one of totally disconnected, complete, star graph, cycle, path and we give a characterization for groups whose intersection graph of cyclic subgroups is one of acyclic, bipartite, $C_3$-free. We show that for a given finite group $G$, $girth(\mathscr I_c (G)) \in \{3, \infty\}$. Moreover, we classify all finite non-cyclic abelian groups whose intersection graph of cyclic subgroups is planar.   Also we determine the independence number, clique cover number of $\mathscr I_c (G)$ for any group $G$. As a consequence, we show that $\mathscr I_c (G)$ is weakly $\alpha$-perfect. Finally, we determine the values of $n$ for which $\mathscr I_c (\mathbb{Z}_n)$ is regular and estimate its domination number.

Throughout this paper $p$, $q$, $r$ denotes the distinct prime numbers.

\section{Main Results}

\begin{thm}\label{cyclic intersection graph 13}
Let $G_1$ and $G_2$ be two groups. If $G_1\cong G_2$, then $\mathscr I_c(G_1)\cong \mathscr I_c(G_2)$.
\end{thm}
\begin{proof}
Let $f:G_1\rightarrow G_2$ be a group isomorphism. Define a map $\psi:V(\mathscr I_c(G_1))\rightarrow V(\mathscr I_c(G_2))$ by $\psi(H)=f(H)$, for every
$H\in V(\mathscr I_c(G_1))$.
It is easy to see that $\psi$ is  a graph isomorphism.
\end{proof}

\begin{rem}
The converse of Theorem~\ref{cyclic intersection graph 13} is not true. For example, let $G_1 \cong \mathbb Z_{p^5}$ and
$G_2 \cong Q_8 =\langle a,b|a^{4}=b^4=1,b^2=a^2,ab=ba^{-1}\rangle$. Here $H_i$, $i=1$, 2, 3, 4 are subgroups of $G_1$ of orders $p^i$ respectively; $H_1$ is a subgroup of $H_i$, $i=2$, 3, 4. $\langle a\rangle$, $\langle b\rangle$, $\langle ab\rangle$, $\langle a^2\rangle$ are
the subgroups of $G_2$; $\langle a^2\rangle$ is a subgroup of $\langle a\rangle$, $\langle b\rangle$, $\langle ab\rangle$.
It follows that $\mathscr I_c(G_1)=K_4=\mathscr I_c(G_2)$, but $G_1\ncong G_2$.
\end{rem}

\begin{thm}\label{cyclic intersection graph 14}
Let $G$ be a group. Then $\mathscr I_c(G)$ is totally disconnected if and only if the order of every element of $G$ is prime.
\end{thm}
\begin{proof}
If the order of every element of $G$ is prime, then obviously proof follows.
Suppose $G$ has an element of composite order, then the subgroup generated by that element is adjacent with its proper cyclic subgroups in $\mathscr I_c(G)$ and so $\mathscr I_c(G)$ is not totally disconnected. This completes the proof.
\end{proof}

\begin{cor}\label{cyclic intersection graph 1}
Let $G$ be a finite group. Then $\mathscr I_c(G)$ is totally disconnected if and only if $G$ is one of $p$-group with exponent $p$ or nilpotent group of
order $p^\alpha q$ or $A_5$.
\end{cor}
\begin{proof} It is proved in \cite{marian} that  in a finite group, every element of it is of prime order if and only if it is one of $p$-group with exponent $p$ or nilpotent group of order
$p^\alpha q$ or $A_5$. This fact together with Theorem~\ref{cyclic intersection graph 14} completes the proof.
\end{proof}
\begin{rem}\label{rem 1}
Tarski Monster's group is an infinite non-abelian group in which every proper subgroup is of order a fixed prime $p$. The existence of such a group was given by Olshanski in~\cite{olshan}. This shows the validity of Theorem~\ref{cyclic intersection graph 14} for infinite groups.
\end{rem}

\begin{thm}\label{cyclic intersection graph 15}
Let $G$ be a group. Then $\mathscr I_c(G)$ is complete if and only if $G$ has a unique subgroup of prime order.
\end{thm}
\begin{proof}
Suppose $G$ has two subgroups of prime order, then they are not adjacent in $\mathscr I_c(G)$. It follows that if $\mathscr I_c(G)$ is complete then $G$ has a unique subgroup of prime order. Conversely, suppose $G$ has a unique subgroup of prime order, then every subgroup of $G$ contains that subgroup. It follows that the intersection of any two subgroups of $G$ is non-trivial and so $\mathscr I_c(G)$ is complete. Hence the proof.
\end{proof}

\begin{cor}\label{cyclic intersection graph 2}
Let $G$ be a finite group. Then $\mathscr I_c(G)$ is complete if and only if $G$ is isomorphic to either $\mathbb Z_{p^\alpha}$ or $Q_{2^\alpha}(\alpha \geq 3)$. In this case, $\mathscr I_c(\mathbb Z_{p^\alpha})\cong K_{\alpha-1}$ and $\mathscr I_c(Q_{2^\alpha})\cong K_{2^{\alpha-2}+\alpha-1}$.
\end{cor}
\begin{proof}
By Theorem~\ref{cyclic intersection graph 15}, $G$ must be a $p$-group. Since $G$ has a unique subgroup of prime order, so by \cite[Proposition 1.3]{scott},
 $G\cong \mathbb Z_{p^\alpha}$ or $Q_{2^\alpha}$. If $G\cong \mathbb Z_{p^\alpha}$, then  $G$ has $\alpha-1$ proper cyclic subgroups and they intersect with each other non-trivially. Thus $\mathscr I_c(G)\cong K_{\alpha-1}$. If $G\cong Q_{2^\alpha}$, then it has $2^{\alpha-2}+\alpha-1$ proper cyclic subgroups and every cyclic subgroups of $G$ contains a unique subgroup of order 2 and so  $\mathscr I_c(Q_{2^\alpha})\cong K_{2^{\alpha-2}+\alpha-1}$.
%
\end{proof}
\begin{rem}\label{rem 1}
$\mathbb{Z}_{p^{\infty}}$ is an example of an infinite group such that $\mathscr I_c (\mathbb{Z}_{p^{\infty}})$ is complete. This shows the validity of Theorem~\ref{cyclic intersection graph 15} for infinite groups.
\end{rem}

In \cite{Chakara}, Chakarabarthy \emph{et al} classified all the cyclic groups whose intersection graph of subgroups are planar. In the next result, we classify all non-cyclic finite abelian groups whose intersection graph of cyclic subgroups are planar. We mainly use the Kuratowski's Theorem to check the planarity of a graph.

\begin{thm}\label{cyclic intersection graph 16}
Let $G$ be a finite non-cyclic abelian group. Then $\mathscr I_c(G)$ is planar if and only if $G$ is isomorphic to one of $\mathbb Z_{p}^n$, $\mathbb Z_4\times \mathbb Z_2$, $\mathbb Z_9\times \mathbb Z_3$, $\mathbb Z_{2q}\times \mathbb Z_2$ or $\mathbb Z_4\times \mathbb Z_4$.
\end{thm}
\begin{proof}
We divide the proof into several cases.

\noindent \textbf{Case 1:} If $G\cong  \mathbb Z_{p}^n, n \geq 1$, then every element of $G$ is prime and so by Theorem~\ref{cyclic intersection graph 1},  $\Gamma_{Ic}(G)$ is totally disconnected.

\noindent \textbf{Case 2:} If $G\cong \mathbb Z_{p^2}\times \mathbb Z_p$.  Here $\langle(1,0)\rangle$, $\langle(1,1)\rangle$, $\ldots$, $\langle(1,p-1)\rangle$,
$\langle(p,0)\rangle$, $\langle(p,1)\rangle$, $\ldots$, $\langle(p,p-1)\rangle$, $\langle(0,1)\rangle$ are the only
proper cyclic subgroups of $G$. Note that $\langle(p,0)\rangle$ is a subgroup of $\langle(1,0)\rangle$, $\langle(1,1)\rangle$, $\ldots$, $\langle(1,p-1)\rangle$
; also no two remaining subgroups intersect non-trivially. Therefore, $\mathscr I_c(G)\cong K_{p+1}\cup \overline{K}_p$. Thus $\mathscr I_c(G)$ is planar if and only if $p=2$, 3.

\noindent \textbf{Case 3:} If $G\cong \mathbb Z_{pq}\times \mathbb Z_p$. Here $G$ has $p+1$ subgroups of order $p$, let them be $H_i$, $i=1$, 2, $\ldots$, $p+1$; unique subgroup of order $q$, say $H$; for each $i=1$, 2, $\ldots$, $p+1$, $HH_i$ is a cyclic subgroups of $G$ of order $pq$. So $H$ is a subgroup of $HH_i$, $i=1$, 2, $\ldots$, $p+1$; for each $i=1$, 2, $\ldots$, $p+1$, $H_i$ is a subgroup of $HH_i$; no two remaining cyclic subgroups intersect non-trivially. So $\mathscr I_c(G)$ is planar if and only if $p=2$.

\noindent \textbf{Case 4:} If $G\cong \mathbb Z_{p^2}\times \mathbb Z_{p^2}$. If $p\geq 5$, then $\mathbb Z_{p^2}\times \mathbb Z_p$ is a subgroup of $G$ and so by case 2, $\mathscr I_c(G)$ is non-planar. If $p=2$, then $\langle (1,0)\rangle$, $\langle (1,2)\rangle$, $\langle (2,0)\rangle$, $\langle (0,1)\rangle$, $\langle (2,1)\rangle$, $\langle (0,2)\rangle$, $\langle (1,1)\rangle$, $\langle (2,2)\rangle$ are the only proper cyclic subgroups of $G$. Also $\langle (2,0)\rangle$ is a proper subgroup of $\langle (1,0)\rangle$, $\langle (1,2)\rangle$; $\langle (0,2)\rangle$ is a subgroup of $\langle (0,1)\rangle$, $\langle (2,1)\rangle$; $\langle (2,2)\rangle$ is a subgroup of $\langle (1,1)\rangle$; no two remaining subgroups intersect. Therefore, $\mathscr I_c(G)\cong 2K_3\cup K_2$, which is planar. If $p=3$, then $G\cong \langle a, b~|~a^9=b^9=1, ab=ba\rangle$. Here $\langle a^3\rangle$ is a subgroup of $\langle a\rangle$, $\langle ab^3\rangle$, $\langle ab^6\rangle$, $\langle a^3b\rangle$. Therefore, these five subgroups adjacent with each other and so they form $K_5$ as a subgraph of $\mathscr I_c(G)$. It follows that $\mathscr I_c(G)$ is non-planar.

\noindent \textbf{Case 5:} If $G\cong \mathbb Z_{p^2q}\times \mathbb Z_p$, then $G$ has the unique subgroup of order $q$, let it be $H$; $G$ has at least three subgroups of order $p^2q$ and at least three subgroups of order $pq$. Therefore, $G$ has at least four cyclic subgroups containing $H$ as a subgroup. Hence $\mathscr I_c(G)$ contains $K_5$ as a subgraph and so $\mathscr I_c(G)$ is non-planar.

\noindent \textbf{Case 6:} If $G\cong \mathbb Z_{p^3}\times \mathbb Z_p$, then $G$ has at least four cyclic subgroups of order $p^3$, $p^3$, $p^2$, $p^2$ respectively and they have a unique subgroup of order $p$ in common. Therefore, these five subgroups are adjacent with each other and they form $K_5$ as a subgraph of $\mathscr I_c(G)$, so $\mathscr I_c(G)$ is non-planar.

\noindent \textbf{Case 7:} If $G\cong \mathbb Z_{p^3}\times \mathbb Z_{p^2}$, then $G$ has a subgroup isomorphic to $\mathbb Z_{p^3}\times \mathbb Z_p$. So by Case 6, $\mathscr I_c(G)$ contains $K_5$ as a subgraph and hence  $\mathscr I_c(G)$ is non-planar.

\noindent \textbf{Case 8:} If $G \cong {\mathbb Z}_{p_1^{\alpha_1}} \times {\mathbb Z}_{p_2^{\alpha_2}} \times \ldots \times {\mathbb Z_{p_k^{\alpha_k}}}$, where
$p_i$'s are primes with at least two $p_i$'s are equal and $\alpha_i\geq1$. If $k\geq 2$, then $G$ has one of the following groups as its subgroup: $\mathbb Z_{p^3}\times \mathbb Z_p$, $\mathbb Z_{p^2q}\times \mathbb Z_p$ or $\mathbb Z_{p^3}\times \mathbb Z_{p^2}$ and so by Cases 5, 6, 7, $\mathscr I_c(G)$ contains $K_5$ as a subgraph. Therefore, $\mathscr I_c(G)$ is non-planar.

Combining all the cases together the proof follows.
\end{proof}

$\mathbb{Z}_p^{\infty}$, $Q_8$ and Tarski Monster groups shows the existence of an infinite abelian group, finite non-abelian group and infinite non-abelian group respectively, whose intersection graph of cyclic subgroups is planar. Now we pose the following:
\begin{problem}
Classify the non-abelian groups and infinite abelian groups whose intersection graph of cyclic subgroups are planar.
\end{problem}

\begin{thm}\label{cyclic intersection graph 3}
Let $G$ be a finite group. Then $\mathscr I_c(G)$ is a star if and only if $G\cong \mathbb Z_{p^3}$.
\end{thm}
\begin{proof} If $G\cong \mathbb Z_{p^3}$, then $\mathscr I_c(G) \cong K_2$, which is a star.

Conversely let $\mathscr I_c(G)$ be a star. Let $H$ be the center vertex and $H_1$, $H_2$, $\ldots$, $H_n$ be the pendent vertices of $\mathscr I_c(G)$. Now
we consider the following cases:

\noindent\textbf{Case 1:} Let $H$ be a subgroup of $H_i$, for all $i=1$, $\ldots$, $n$. Then $H$, $H_1$, $H_2$ forms $C_3$ as a
subgraph of $\mathscr I_c(G)$, which is not possible. So the only possibility is $n=1$ and $H$ is a subgroup of $H_1$.

\noindent\textbf{Case 2:} Let $H_i$ be a subgroup of $H$, for all $i=1$, 2, $\ldots$, $n$. Then we show that $|H_i|=p$, for all $i$.  For otherwise, $|H_i|=q$ or $m$ for some $i$, where $m$ is a composite number. If $|H_i|=m$, then $H$, $H_i$ and a proper subgroup of $H_i$ forms $C_3$ as a subgraph of $\mathscr I_c(G)$, which is a contradiction to our assumption.
If $|H_i|=q$, then $|H|=pq$ and so $|G|=p^{\alpha}q^{\beta}$, $\alpha + \beta \geq 3$ or $pqk$, $k >1$. But $|G|=p^{\alpha}q^{\beta}$ is not possible, since it has a subgroup of order $p^2$. $|G|=pqk$ is also not possible, since it has a subgroup of order other than $p$ and $q$.

Thus we have $|H_i|=p$, for all $i$ and so $|H|=p^2$. It follows that $|G|=p^\alpha$. Since $G$ has a unique subgroup of order $p$, so by \cite[Proposition 1.3]{scott}, $G\cong \mathbb Z_{p^\alpha}$ or $Q_{2^\alpha}$. But $G$ can not be isomorphic to $Q_{2^\alpha}$, since $Q_{2^\alpha}$ contains more than one subgroup of order 4. So by Corollary~\ref{cyclic intersection graph 2}, we must have $G\cong \mathbb Z_{p^3}$. This completes the proof.
\end{proof}

\begin{thm}\label{cyclic intersection graph 4}
Let $G$ be a finite group. Then $\mathscr I_c(G)\cong P_n$ if and only if $n=1$ and $G\cong \mathbb Z_{p^3}$.
\end{thm}
\begin{proof}If $G\cong \mathbb Z_{p^3}$, then $\mathscr I_c(G) \cong K_2$, which is a path.

Conversely, let $\mathscr I_c(G)$ be a path $H_1-H_2-\ldots -H_r$. Let $r\geq 3$. Firstly note that for distinct $i$, $j$, $k$, $H_i\cap H_j=H_k$ is not possible, since these subgroups forms $C_3$ as a subgraph of $\mathscr I_c(G)$. So for each $i=2$, 3, $\ldots$, $r-1$, the subgroups $H_{i-1}$, $H_i$, $H_{i+1}$ must satisfy one of the following: (i) $H_{i-1}\subseteq H_i\subseteq H_{i+1}$, (ii) $H_{i-1}\supseteq H_i\supseteq H_{i+1}$, (iii) $H_{i-1}\subseteq H_i\supseteq H_{i+1}$. But first two cases are not possible, since these three subgroups forms $C_3$ as a subgraph of $\mathscr I_c(G)$. So the only possibility is the case (iii). If $r\geq 4$, then $H_1\subseteq H_2\supseteq H_3\subseteq H_4$, so that $H_2$, $H_3$, $H_4$ forms $C_3$ as a subgraph of $\mathscr I_c(G)$, which is a contradiction. Thus $r\leq 3$. If $r=3$, then $\mathscr I_c(G)\cong K_{1,2}$ is a star, which is not possible by Theorem~\ref{cyclic intersection graph 3}. So we must have $r=2$ and it follows that $G\cong \mathbb Z_{p^3}$.
\end{proof}

\begin{thm}\label{cyclic intersection graph 5}
Let $G$ be a finite group. Then $\mathscr I_c(G)\cong C_n$ if and only if $n=3$ and $G\cong \mathbb Z_{p^4}$.
\end{thm}

\begin{proof}If $G\cong \mathbb Z_{p^4}$, then $\mathscr I_c(G) \cong C_3$.

Conversely, Suppose that $\mathscr I_c(G)$ is a cycle $H_1-H_2-\cdots -H_r-H_1$. Let $2<r\geq 4$. Then by the same argument as in the proof of Theorem~\ref{cyclic intersection graph 4}, we can show that $r\leq 3$.
If $r=3$, then by Theorem~\ref{cyclic intersection graph 2}, $\mathscr I_c(G)\cong K_3$ if and only if $G\cong \mathbb Z_{p^4}$.
\end{proof}

\begin{cor}\label{cyclic intersection c1}
Let $G$ be a finite group. Then $girth(\mathscr I_c(G)) \in \{3  , \infty \}$.	
\end{cor}
\begin{proof}
In the proof of Theorem~\ref{cyclic intersection graph 5}, we observed that if $\mathscr I_c(G)$ contains a cycle of length at least four, then it must contains $C_3$ as a subgraph, so its girth is 3. In the remaining cases, $\mathscr I_c(G)$ is acyclic and so its girth is infinity. Hence the proof.
\end{proof}

\begin{thm}\label{cyclic intersection graph 7}
Let $G$ be a finite group. Then the following are equivalent:
\begin{enumerate}[{\normalfont (1)}]
\item $G$ has a maximal cyclic subgroup of order one of $p$, $p^2$ or $pq$ and no two cyclic subgroups of order $p^2$ or $pq$ have a non-trivial intersection;
\item $\mathscr I_c(G)$ is acyclic;
\item $\mathscr I_c(G)$ is bipartite;
\item $\mathscr I_c(G)$ is $C_3$-free.
\end{enumerate}
\end{thm}
\begin{proof} Clearly $(1) \Rightarrow (2) \Rightarrow (3) \Rightarrow (4)$.

\noindent  Now assume that $\mathscr I_c(G)$ is acyclic.
Suppose either $G$ has a maximal cyclic subgroup of order $m> p, p^2$, $pq$ and two maximal cyclic subgroups intersect non-trivially, then $\mathscr I_c(G)$ contains $C_3$, which is a contradiction to our assumption. This prove the part $(2) \Rightarrow (1)$.

\noindent Suppose $\mathscr I_c(G)$ is $C_3$-free,
then by Corollary~\ref{cyclic intersection c1}, $girth(\mathscr I_c(G))=\infty$ and so $\mathscr I_c(G)$ is acyclic. This prove the part $(4) \Rightarrow (2)$. Hence the proof.
\end{proof}

\begin{thm}\label{cyclic intersection graph 8}
Let $G$ be a group. Then $\alpha(\mathscr I_c(G))=m$, where $m$ is the number of prime order subgroups of $G$.
\end{thm}
\begin{proof}
Let $\mathscr{A}$ be the set of all prime order subgroups of $G$. Clearly $\mathscr{A}$ is a maximal independent set in $\mathscr I_c(G)$. We show that $\alpha(\mathscr I_c(G))=|\mathscr A|$. Let $\mathscr{B}$ be a another one independent set $\mathscr I_c(G)$, which is different from $\mathscr{A}$. Then $\mathscr{B}$
has at least one cyclic subgroup of $G$ of composite order, say $H$. If $|H|=p^2$, then $H$  has a unique subgroup of order $p$. If $|H|\neq p^2$, then $H$  has at least two subgroups of prime orders. In either case, these prime order subgroups are adjacent with $H$ in $\mathscr I_c(G)$.  So for each vertex in $\mathscr{B}$, there corresponds at least one vertex in $\mathscr{A}$. It follows that $|\mathscr{B}|\leq |\mathscr{A}|$. Hence the proof.
\end{proof}


\begin{thm}\label{cyclic intersection graph 300}
Let $G$ be a group. Then
$\Theta(\mathscr I_c(G))=m$, where $m$ is the number of prime order subgroups of $G$.
\end{thm}
\begin{proof} Let $|G|=p_1^{\alpha_1}p_2^{\alpha_2}\ldots p_k^{\alpha_k}$, where $p_i$'s are distinct primes, $\alpha_i\geq 1$.
	For each $i=1$, 2, $\ldots$, $k$, let $t_i$ be the number of subgroups of order $p_i$ and let $H(j,p_i)$, $j=1$, 2, $\ldots$, $t_i$ be a
	subgroup of $G$
	of order $p_i$. For each $j=1$, 2, $\ldots$, $t_i$, let $\mathscr S(j,p_i)$ be the set of all  proper cyclic subgroups of $G$ having $H(j,p_i)$ in common.
	Clearly $\mathscr S(j, p_i)$ forms a clique in $\mathscr{I}_c(G)$ and $\mathscr S:=\{\mathscr S(j,p_i)~|~i=1$, 2, $\ldots$, $k$ and $j=1$, 2, $\ldots$, $t_i\}$
	forms a clique cover of $\mathscr{I}_c(G)$ with  $|\mathscr S|=t_1+t_2+\cdots +t_k$. Therefore, $\Theta(\mathscr{I}_c(G))\leq |\mathscr S|$. Let $\mathscr{T}$ be a clique cover of $\mathscr{I}_c(G)$ such that $\Theta(\mathscr{I}_c(G))=|\mathscr{T}|$. If $|\mathscr{T}|<|\mathscr S|$, then by pigeonhole principle, $\mathscr{T}$ has a clique which contains atleast two subgroups, say $H(j,p_i)$, $H(l,p_r)$, for some $i\neq r$,
	$j\in \{1$, 2, $\ldots$, $t_i\}$, $l\in \{1$, 2, $\ldots$, $t_r\}$, which is not possible, since $H(j,p_i)$ and $H(l,p_r)$ are not adjacent in $\mathscr{I}_c(G)$. So $|\mathscr{T}|=|\mathscr S|=$ the number of prime order subgroups of $G$. A similar argument also works when $G$ is infinite.
\end{proof}

\begin{rem} In \cite{zelinka}, Zelinka proved that for any group $G$, $\alpha(\mathscr{I}(G))=m$, where $m$ is the number of prime order proper subgroups of $G$.
In~\cite{raj_5}, the authors showed that $\Theta(\mathscr I(G))=\alpha(\mathscr{I}(G))$. It is interesting to note that by Theorems~\ref{cyclic intersection graph 8} and~\ref{cyclic intersection graph 300}, we have $\alpha(\mathscr{I}(G))=\alpha(\mathscr I_c(G))$ and $\Theta(\mathscr I(G)) = \Theta(\mathscr I_c(G))$.
\end{rem}
\begin{cor}\label{301}
Let $G$ be a group. Then $\mathscr I_c(G)$ is weakly $\alpha$-perfect.
\end{cor}
\begin{proof}
Proof follows from Theorems~\ref{cyclic intersection graph 8} and \ref{cyclic intersection graph 300}.
\end{proof}

\begin{thm}\label{intersecting graph t24}
 $\mathscr I_c(\mathbb Z_n)$ is regular if and only if $n=p^\alpha$, $\alpha \geq 2$.
\end{thm}
\begin{proof}
Let $n=p_1^{\alpha_1}p_2^{\alpha_2}\ldots p_k^{\alpha_k}$, where ${p_i}'$s are distinct primes, $\alpha_i \geq 1$. We need to consider the following cases.

\noindent \textbf{Case 1:} Let $k\geq 2$. It is shown in \cite[p. 5387]{Chakara}, that for any subgroup $H$ of $\mathbb Z_n$, if $|H|=p_{i_1}^{t_{i_1}}p_{i_2}^{t_{i_2}}\ldots p_{i_r}^{t_{i_r}}$,
then $\deg(H)=[\tau(n)-\displaystyle\prod_{\substack{j=1\\ j\notin \{i_1$, $i_2$, $\ldots$, $i_r\}}}^n(\alpha_i+1)]-2$;
if $|H|=p_1^{\alpha_{i_1}}p_2^{\alpha_{i_2}}\ldots p_k^{\alpha_{i_k}}$, then $\deg(H)=\tau(n)-3$.

Now, let $K$ be the subgroup of $\mathbb Z_n$ of order $p_r^{\alpha_r}$.
If some $\alpha_i>1$, let $A$ be a subgroup of $\mathbb Z_n$ of order $p_1^{\alpha_{i_1}}p_2^{\alpha_{i_2}}\ldots p_k^{\alpha_{i_k}}$, where
$0\leq \alpha_{i_t}\leq \alpha_t$.
Then $\deg(K)=\tau(n)-2-\displaystyle\prod_{\substack{i=1\\ i\neq r}}^k (\alpha_i+1)$ and $\deg(A)=\tau(n)-3$.
So $\deg(A)>\deg(K)$. Therefore, $\mathscr I_c(\mathbb Z_n)$ is not regular.
If $\alpha_i=1$ for every $i=1$, 2, $\ldots$, $k$,
then let $B$ be the subgroup of $\mathbb Z_n$ of order $p_1p_2\ldots p_{i-1}p_{i+1}\ldots p_k$. Then $\deg(B)=2^k-4$ and
$\deg(K)=2^{k-1}-2$. But $\deg(B)=2^k-4=2(2^{k-1}-2)>2^{k-1}-2=\deg(K)$. Therefore, $\mathscr I_c(G)$ is not regular.

\noindent \textbf{Case 2:} $k=1$, then by Theorem~\ref{cyclic intersection graph 2}, $\mathscr I_c(\mathbb Z_n)\cong K_{\alpha-1}$, which is regular.

Combining the above two cases we get the result.
\end{proof}

\begin{thm}\label{intersecting graph t22}
 Let $n=p_1^{\alpha_1}p_2^{\alpha_2}\ldots p_k^{\alpha_k}$, where ${p_i}'$s are distinct primes, $\alpha_i \geq 1$.
Then
\begin{align*}
\gamma(\mathscr I_c(\mathbb Z_n))&=\left\{
\begin{array}{ll}
		1, & \mbox{if~ } \alpha_i>1~\mbox{for some}~i; \\
		2, & \mbox{if~ } \alpha_i=1~\mbox{for every}~i.
	\end{array}
\right.
\end{align*}
\end{thm}
\begin{proof}
 If $\alpha_i>1$, for some $i\in\{1$, 2, $\ldots$, $k\}$, then $\mathbb Z_n$ has a subgroup of order $p_1p_2\ldots p_k$, and
it intersects non-trivially with every other proper subgroups of $\mathbb Z_n$,
so $\gamma(\mathscr I_c(\mathbb Z_n))=1$. If $\alpha_i=1$, for every $i=1$, 2, $\ldots$, $n$, then $\mathbb Z_n$ has subgroups
of orders $p_1p_2\ldots p_{k-1}$ and $p_2p_3\ldots p_k$
respectively, and these
two subgroups forms a dominating set in $\mathscr I_c(\mathbb Z_n)$. So we have $\gamma(\mathscr I_c(\mathbb Z_n))\leq 2$. Suppose $\gamma(\mathscr I_c(\mathbb Z_n))<2$, then there exist proper
a subgroup,
which has non-trivial intersection with every other proper subgroups of $\mathbb Z_n$; but it is not possible, since $\alpha_i=1$,
for every $i=1$, 2, $\ldots$, $n$.
Therefore, $\gamma(\mathscr I_c(\mathbb Z_n))=2$.
\end{proof}

\end{document}